\documentclass[11pt,a4paper]{amsart}
\usepackage{amssymb,amsmath,epsfig,graphics,mathrsfs}

\usepackage{fancyhdr}
\usepackage{hyperref}
\hypersetup{
 colorlinks   = true,
 urlcolor     = blue,
 linkcolor    = blue,
 citecolor   = red ,
 bookmarksopen=true
}

\usepackage{amsmath}
\usepackage{amsfonts}
\usepackage{amssymb}
\usepackage{amsthm}
\usepackage{epsfig,graphics,mathrsfs}
\usepackage{graphicx}
\usepackage{dsfont}

\usepackage[usenames, dvipsnames]{color} 

\usepackage{hyperref}

\newcommand*\MSC[1][1991]{\par\leavevmode\hbox{%
\textit{#1 Mathematical subject classification:\ }}}
\newcommand\blfootnote[1]{%
  \begingroup
  \renewcommand\thefootnote{}\footnote{#1}%
  \addtocounter{footnote}{-1}%
  \endgroup
}

\def \phi {\varphi}

\def \R {\mathbb{R}}

\def \vf{\varphi}

\newcommand{\Rn}{\mathbb R^n}

\newcommand{\Hn}{\mathbb H^n}
\newcommand{\Ho}{\mathbb H^1}

\newcommand{\p}{\partial}

\newcommand{\bG}{\mathbb {G}}
\newcommand{\bg}{\mathfrak g}

\newcommand{\la}{\lambda}

\numberwithin{equation}{section}

\newcommand{\beq}{\begin{equation}}
\newcommand{\bea}[1]{\begin{array}{#1} }
\newcommand{\eeq}{ \end{equation}}
\newcommand{\ea}{ \end{array}}

\newcommand{\ve}{\varepsilon}

\newcommand{\sij}{\sum_{i,j=1}^m}
\newcommand{\nh}{\nabla_H}
\newcommand{\sul}{\Delta_H}
\newcommand{\nhh}{\tilde{\nabla}_H}

\newtheorem{theorem}{Theorem}[section]
\newtheorem{lemma}[theorem]{Lemma}
\newtheorem{proposition}[theorem]{Proposition}
\newtheorem{corollary}[theorem]{Corollary}
\newtheorem{remark}[theorem]{Remark}

\numberwithin{equation}{section}

\begin{document}

\title[A note on monotonicity and Bochner formulas, etc.]{A note on monotonicity and Bochner formulas in Carnot groups}

\blfootnote{\MSC[2020]{35R03, 35H10, 31C05}}
\keywords{Bochner formulas. Monotonicity formulas. Almgren type frequency}

\date{}

\begin{abstract}
In this note we prove two monotonicity formulas for solutions of $\Delta_H f = c$ and $\Delta_H f - \p_t f = c$ in Carnot groups. Such formulas involve the right-invariant \emph{carr\'e du champ} of a function and they are false for the left-invariant one. The main results, Theorems \ref{T:babyACF} and \ref{T:babyC}, display a resemblance with two deep monotonicity formulas respectively due to Alt-Caffarelli-Friedman for the standard Laplacian, and to Caffarelli for the heat equation. In connection with this aspect we ask the question whether an ``almost monotonicity" formula be possible. In the last section we discuss the failure of the nondecreasing monotonicity of an Almgren type functional.    
\end{abstract}

\author{Nicola Garofalo}

\address{Dipartimento d'Ingegneria Civile e Ambientale (DICEA)\\ Universit\`a di Padova\\ Via Marzolo, 9 - 35131 Padova,  Italy}
\vskip 0.2in
\email{nicola.garofalo@unipd.it}

\thanks{The author is supported in part by a Progetto SID: ``Non-local Sobolev and isoperimetric inequalities", University of Padova, 2019.}

\maketitle


\section{Introduction and statement of the results}\label{S:intro}

Monotonicity formulas play a prominent role in analysis and geometry. They are often employed in the blowup analysis of a given problem to derive information on the regularity of the solutions, or on their global configurations. In this note we prove two monotonicity formulas, Theorems \ref{T:babyACF} and \ref{T:babyC} below, in the geometric setup of  Carnot groups. While these Lie groups display some superficial similarities with the Euclidean framework, they are intrinsically non-Riemannian (see E. Cartan's seminal address \cite{Ca}), and the counterpart of many classical results simply fails to be true. Our monotonicity results fall within this category. They are false, in general, if in their statements one replaces the right-invariant \emph{carr\'e du champ} with the ``more natural" left-invariant one. 

Our interest in monotonicity formulas stems from our previous joint works \cite{DGP, DGS} on some nonholonomic free boundary problems suggested to us by people in mechanical engineering and
robotics at the Johns Hopkins University. In \cite{DGS} the optimal interior regularity $\Gamma^{1,1}_{loc}$ of the solution of a certain obstacle problem was established. While such result guarantees the boundedness of the second horizontal derivatives $X_i X_j f$ of the solution, it falls short of implying their continuity. This critical information was subsequently established in \cite{DGP} in the framework of Carnot groups of step $k=2$, where it was also proved that, under a suitable thickness assumption, the free boundary is remarkably a $C^{1,\alpha}$ non-characteristic hypersurface, suggesting a connection with the sub-Riemannian Bernstein problem, see \cite{DGNP}. The key  idea in \cite{DGP} was the systematic use of the right-invariant derivatives in the study of a
left-invariant free boundary problem\footnote{In harmonic analysis and PDEs the use of right-invariant derivatives in left-invariant problems had already appeared in the works \cite{MPR} and \cite{BL}.}. This leads us to the main theme of this note.

Given a Carnot group $(\bG,\circ)$, we denote the left-translation operator by $L_g(g') = g \circ g'$ and with $dL_g$ its differential. The right-translation will be denoted by $R_g(g') = g' \circ g$, and its differential by $dR_g$. If we fix an orthonormal basis $\{e_1,...,e_m\}$ of the horizontal layer $\bg_1$, then we can define respectively left- and right-invariant vector fields by the formulas
\[
X_i(g) = dL_g(e_i),\ \ \ \ \ \ \ \tilde X_i(g) = dR_g(e_i).
\]
More in general, for any $\zeta \in \bg$ we respectively indicate with $Z$,
and $\tilde Z$ the left- and right-invariant vector fields on
$\bG$ defined by the Lie formulas 
\begin{equation}\label{i1}
Zf(g)= \frac{d}{dt}\ f(g\circ \exp(t\zeta))\big|_{t=0},\quad\quad\quad \ \tilde{Z} f(g) = \frac{d}{dt}\ f(
\exp(t\zeta)\circ g)\big|_{t=0} .
\end{equation}
For any $\eta , \zeta \in \bg$, for the corresponding
vector fields on $\bG$ we have the following simple, yet basic, commutation identities 
\begin{equation}\label{i2}
[Y , \tilde Z] = [\tilde Y , Z] = 0.
\end{equation}
Such identities can be easily verified using \eqref{i1} and the Baker-Campbell-Hausdorff formula. 
From \eqref{i2} we have in particular $[X_i,\tilde X_j] = 0$, for $i, j=1,...,m$. Given a function $f\in C^1(\bG)$ we will respectively denote by
\begin{equation}\label{nabla}
|\nh f|^2 = \sum_{i=1}^m (X_i f)^2,\ \ \ \ \ \ |\nhh f|^2 = \sum_{i=1}^m (\tilde X_i f)^2,
\end{equation}
the left- and right-invariant \emph{carr\'e du champ} of $f$. If we indicate with $e\in \bG$ the group identity, since $X_i(e) = \tilde X_i(e)$ for $i=1,...,m$, we have
\begin{equation}\label{eq}
|\nh f(e)|^2 = |{\tilde{\nabla}}_H f(e)|^2.
\end{equation}
But the two objects in \eqref{nabla} are substantially different, except in the trivial situation in which the function $f$ depends exclusively on the horizontal variables, see for instance \eqref{difference} below. 

The left-invariant horizontal Laplacian relative to $\{e_1,...,e_m\}$ is defined on a function $f\in C^2(\bG)$ by the formula
\begin{equation}\label{L}
\sul f = \sum_{i=1}^m X_i^2 f.
\end{equation}
This operator is hypoelliptic thanks to the result in \cite{Ho}. When the step of the stratification of $\bg$ is $k=1$, then the group is Abelian and $\sul = \Delta$ is the standard Laplacian. However, in the genuinely sub-Riemannian situation $k>1$, the differential operator $\sul$ fails to be elliptic at every point of the ambient space $\bG$. We say that a function $f\in C^2(\bG)$ is subharmonic (superharmonic) if $\sul f \ge 0 (\le 0)$. We say that $f$ is harmonic if it is both sub- and superharmonic. These notions can be extended in the weak variational sense in a standard fashion. 

Let now $\rho$ be the pseudo-gauge, centred at $e$, defined in (2.7) of \cite{GR}. Let $B_r = \{g\in \bG\mid \rho(g)<r\}$ and $S_r = \p B_r$. Let $Q>N$ indicate the homogeneous dimension of $\bG$ associated with the natural anisotropic dilations ($Q = N$ only in the Abelian case $k=1$). Given a function $f\in C(B_1)$, and a number $0<\alpha <Q$, we consider the functional 
\begin{equation}\label{JG}
\mathscr M_\alpha(f,r) = \frac{1}{r^\alpha} \int_{B_r} \frac{f(g)}{\rho(g)^{Q-\alpha}} |\nabla_H\rho(g)|^2 dg.
\end{equation}
It is easy to verify (see the opening of Section \ref{S:proof}) that there exists a universal number $\omega_\alpha>0$ such that for every $r>0$ one has
\begin{equation}\label{omega}
\frac{1}{r^\alpha} \int_{B_r} \frac{1}{\rho^{Q-\alpha}} |\nabla_H\rho(g)|^2 dg = \omega_\alpha.
\end{equation}
As a consequence, one has
\begin{equation}\label{zero}
\underset{r\to 
0^+}{\lim} \mathscr M_\alpha(f,r) = \omega_\alpha f(e).
\end{equation}
We have the following. 

\begin{theorem}[Monotonicity formula]\label{T:babyACF}
Let $f$ be a solution of $\sul f = c$ in $B_1$, for some $c\in \R$. Then for any $0<\alpha<Q$ the functional 
\begin{equation}\label{babyACF}
\mathscr D_\alpha(f,r) = \frac{1}{r^\alpha} \int_{B_r} \frac{|{\tilde{\nabla}}_H f(g)|^2}{\rho(g)^{Q-\alpha}} |\nabla_H\rho(g)|^2 dg
\end{equation}
is nondecreasing in $(0,1)$. Moreover, we have for every $r\in (0,1)$
\begin{equation}\label{grades}
\omega_\alpha |\nh f(e)|^2 \le  \mathscr D_\alpha(f,r).
\end{equation}
\end{theorem}

As we have mentioned, Theorem \ref{T:babyACF} ceases to be true, and in the worse possible way, if in the definition \eqref{babyACF} of the functional $\mathscr D_\alpha(f,r)$ we replace the right-invariant \emph{carr\'e du champ} $|\nhh f|^2$ with the left-invariant one $|\nabla_H f|^2$. 

Our next result, Theorem \ref{T:babyC}, should be seen as a parabolic companion of Theorem \ref{T:babyACF}. Denote by $p(g,g',t) = p(g',g,t)$ the smooth, symmetric, strictly positive heat kernel constructed by Folland in \cite{Fo}. Given a reasonable function $\vf$, the solution of the Cauchy problem $\p_t f - \sul f = 0$ in $\bG\times (0,\infty)$, $f(g,0) = \vf(g)$, is given by
\[
f(g,t) = P_t \vf(g) = \int_{\bG} p(g,g',t) \vf(g') dg'.
\]
\begin{theorem}[Heat monotonicity formula]\label{T:babyC}
Let $f$ be a solution of $\p_t f - \sul f = c$ in $\bG\times (-1,0]$, for some $c\in \R$, and suppose that there exist $A, \alpha>0$ such that such that for every $g\in \bG$ and $t\in [-1,0]$ one has 
\begin{equation}\label{gaussian}
|f(g,t)| \le  A\ e^{\alpha d(g,e)^2},
\end{equation} 
where we have denoted by $d(g,g')$ the control distance in $\bG$ associated with the horizontal layer $\bg_1$ of the Lie algebra. Then, there exists $T = T(\alpha)>0$ such that the functional
\begin{equation}\label{babyC}
\mathscr I(f,t) = \frac 1t \int_{-t}^0 \int_{\bG} |\nhh f(g,s)|^2 p(g,e,-s) dg ds  
\end{equation}
is nondecreasing in $t\in (0,T)$. Furthermore, we have for every $t\in (0,T)$
\begin{equation}\label{grades}
|\nh f(e,0)|^2 \le  \mathscr I(f,t).
\end{equation}
\end{theorem}
Similarly to Theorem \ref{T:babyACF}, also Theorem \ref{T:babyC} fails in general if in the definition of $\mathscr I(f,t)$ we replace $|\nhh f|^2$ with $|\nh f|^2$. This failure is caused in both cases by the fact that in sub-Riemannian geometry it is not true in general that if $\sul f = c$, then $|\nabla_H f|^2$ is subharmonic! There exist harmonic functions $f$ such that $|\nabla_H f|^2$ is superharmonic on large regions of $\bG$! For instance, consider in the Heisenberg group $\mathbb H^1$ (for this Lie group see the discussion following Corollary \ref{C:B} below) the harmonic function\footnote{For the reader's understanding, we mention that $f = P_3 - P_1$ where $P_3(x,y,\sigma) = x^3 + xy^2 - 8y\sigma$ is a solid harmonic of degree three in $\mathbb H^1$, and $P_1(x,y,\sigma) = x$ is a solid harmonic of degree one. Such solid harmonics were constructed by Greiner, see \cite[p.387]{Gr}} 
\begin{equation}\label{fmia}
f(x,y,\sigma) = x^3 + xy^2 - 8y\sigma - x.
\end{equation}
A  calculation shows that
\begin{equation}\label{ted}
\sul(|\nh f|^2)(x,y,\sigma) = 176 x^2 + 432 y^2 - 32 \le 432 |z|^2 - 32\le 0,
\end{equation}
provided that the point $g = (x,y,\sigma)$ belongs to the infinite cylinder $|z|^2 \le \frac{2}{27}$ in $\mathbb H^1$. Another example is provided by the harmonic function \eqref{f} below. 
In contrast with \eqref{ted}, as a consequence of our right-invariant Bochner identity in Proposition \ref{P:Br} below, we show the crucial fact that in any Carnot group $\bG$ a solution of $\sul f = c$ always satisfies globally
\[
\sul(|\nhh f|^2) \ge 0.
\] 

The reader who is versed in free boundary problems will recognise in Theorems \ref{T:babyACF} and \ref{T:babyC} a resemblance with two deep monotonicity formulas respectively due to Alt-Caffarelli-Friedman (ACF henceforth) for the standard Laplacian \cite[Lemma 5.1]{ACF}, and to Caffarelli for the classical heat equation \cite[Theor. 1]{Caf93}. The former states that if one is given in the Euclidean ball $B_1\subset \Rn$ two continuous functions $f_\pm$ satisfying
\[
f_\pm \ge 0,\ \ \ \Delta f_\pm \ge 0,\ \ \ f_+ \cdot f_- = 0,\ \ \ f_+(0) = f_-(0) = 0,
\]
then the ACF functional 
\begin{equation}\label{acf}
\Phi(f_+,f_-,r) = \frac{1}{r^4} \int_{B_r} \frac{|\nabla f_+|^2}{|x|^{n-2}} dx  \int_{B_r} \frac{|\nabla f_-|^2}{|x|^{n-2}} dx
\end{equation}
is nondecreasing for $0<r<1$. This monotonicity formula plays a critical role in free boundary problems with a double phase, see e.g. \cite{CS} and \cite{PSU}, where it is used to show that: (a) $\underset{r\to 0^+}{\lim} \Phi(f_+,f_-,r)$ exists, and (b) such limit is less than $\Phi(f_+,f_-,1)$. When $f_{\pm}$ are smooth and their supports intersect along a hypersurface $\Sigma$ through the origin, then the $\underset{r\to 0^+}{\lim} \Phi(f_+,f_-,r)$ is the product of the normal derivatives to $\Sigma$ of $f_{\pm}$ in $x = 0$. Specialised to the case $\bG = \Rn$ and $\alpha = 2$ the functional \eqref{babyACF} in our Theorem \ref{T:babyACF} is precisely half of the ACF functional in \eqref{acf}. Similarly, the functional \eqref{babyC} in our Theorem \ref{T:babyC} is half of the Caffarelli functional for the heat equation in \cite{Caf93}. 

In light of Theorems \ref{T:babyACF} and \ref{T:babyC} above, and with potential applications to nonholonomic free boundary problems with two phases in mind, it is tempting to propose the following \emph{conjecture}:

\medskip

\begin{itemize}
\item[\textbf{(1)}] \emph{Let $\bG$ be a Carnot group and suppose that in $B_1\subset \bG$ we have two continuous functions $f_\pm$ satisfying}
\[
f_\pm \ge 0,\ \ \ \sul f_\pm = - 1,\ \ \ f_+ \cdot f_- = 0\ \ \ f_+(e) = f_-(e) = 0.
\]
\emph{Prove (or disprove?) that the functional}
\begin{equation}\label{D2}
\mathscr D_2(f_+,f_-,r) = \frac{1}{r^4} \mathscr D_2(f_+,r) \mathscr D_2(f_-,r)
\end{equation}
\emph{satisfies the following bound for $0<r<1$} 
\begin{equation}\label{cjk}
\mathscr D_2(f_+,f_-,r) \le C \left\{1 + \mathscr D_2(f_+,1) + \mathscr D_2(f_-,1)\right\}.
\end{equation}
\item[\textbf{(2)}] \emph{Let $\bG$ be a Carnot group and suppose that we have two continuous functions $f_\pm$ satisfying in $\bG\times (-1,0]$}
\[
f_\pm \ge 0,\ \ \ (\sul -\p_t) f_\pm = - 1,\ \ \ f_+ \cdot f_- = 0, \ \ \ f_+(e,0) = f_-(e,0) = 0,
\]
\emph{and with moderate growth at infinity. Prove (or disprove?) that the functional}
\[
\mathscr I(f_+,f_-,t) = \frac{1}{t^2} \mathscr I(f_+,t) \mathscr I(f_-,t)
\]
\emph{satisfies the following bound for $0<t<1$} 
\begin{equation}\label{ck}
\mathscr I(f_+,f_-,t) \le C \left\{1 + \mathscr I(f_+,1) + \mathscr I(f_-,1)\right\}.
\end{equation}
\end{itemize}

\medskip

Besides the circumstantial  evidence provided by Theorems \ref{T:babyACF} and \ref{T:babyC}, this conjecture is inspired by the Caffarelli, Jerison and Kenig powerful modification of the ACF monotonicity formula in which the assumption $\Delta f_\pm \ge 0$ is replaced by the weaker $\Delta f_\pm \ge -1$, and which does not have any ``monotonicity" left in its statement, see \cite[Theor. 1.3]{CJK}. While when $\bG = \Rn$ a uniform bound such as \eqref{cjk} appears only remotely connected to the ACF monotonicity \eqref{acf}, it does nonetheless lead to the Lipschitz continuity of the solutions, and once this is known than one can go full circle and restore monotonicity, as shown in \cite{CJK}. We also cite \cite{Sha} for various applications of the Caffarelli-Jerison-Kenig result to the $C^{1,1}$ regularity in free boundary problems, and \cite{CK} and \cite{MP} for some remarkable parabolic versions of the monotonicity formula \eqref{acf} and the ``almost monotonicity'' formulas \eqref{cjk} and \eqref{ck}. 

We reiterate that all the functionals in the above conjectured \eqref{cjk} and \eqref{ck} involve the right-invariant \emph{carr\'e du champ} $|\nhh f_\pm|^2$. In this respect, we mention that in the recent papers \cite{FF1, FF2} the authors have proposed in the Heisenberg group $\Hn$ a nondecreasing monotonicity formula in which the ACF functional is substituted by the following one containing the left-invariant \emph{carr\'e du champ} of the functions $f_+$ and $f_-$
\begin{equation}\label{ff}
\mathfrak I(f_+,f_-,r) = \frac{1}{r^4} \int_{B_r} \frac{|\nh f_+(g)|^2}{\rho(g)^{Q-2}} dg  \int_{B_r} \frac{|\nh f_-(g)|^2}{\rho(g)^{Q-2}} dg.
\end{equation}
 The same authors have quite recently recognised in \cite[Theor. 1.1]{FF3} that their conjecture cannot be possibly true. In $\mathbb H^1$ with coordinates $g = (x,y,\sigma)$ they consider the following harmonic function (see the footnote to \eqref{fmia} above) 
 \begin{equation}\label{f}
f(x,y,\sigma) = x + 6 y \sigma - x^3,
\end{equation}
and with rather long calculations they show that 
\[
r\ \longrightarrow\ \frac{1}{r^2} \int_{B_r} \frac{|\nh f(g)|^2}{\rho(g)^{Q-2}} dg
\]
is nonincreasing as $r\in (0,r_0)$ for a sufficiently small $r_0>0$. Since on the function \eqref{f} (but \eqref{fmia} would equally work) each half of \eqref{ff} is invariant with respect to the change of variable $(x,y,\sigma) \to (-x,-y,\sigma)$ (see \eqref{nh} below), they infer that
\[
\frac{1}{r^2} \int_{B_r} \frac{|\nh f_+(g)|^2}{\rho(g)^{Q-2}} dg = \frac{1}{r^2} \int_{B_r} \frac{|\nh f_-(g)|^2}{\rho(g)^{Q-2}} dg,
\]
which shows that 
\[
r\ \to\ \mathfrak I(f_+,f_-,r) = \left(\frac{1}{r^2} \int_{B_r} \frac{|\nh f_+(g)|^2}{\rho(g)^{Q-2}} dg\right)^2 = \frac 14 \left(\frac{1}{r^2} \int_{B_r} \frac{|\nh f(g)|^2}{\rho(g)^{Q-2}} dg\right)^2
\]
is nonincreasing (instead of nondecreasing) on $(0,r_0)$, thus disproving their own conjecture.
We emphasise that, instead, neither of the functions \eqref{fmia}, \eqref{f} produces a counterexample to our conjecture above. The next result gives a perspective on the negative example \eqref{f} which is somewhat different from that in \cite{FF3}.

\begin{proposition}\label{P:ce}
For the harmonic function \eqref{f} one has 
\[
\sul(|\nh f|^2)(x,y,\sigma) \le 0,
\]
for every $(x,y,\sigma)\in \mathbb H^1$ such that $x^2 + y^2 \le \frac 19$. As a consequence, the left-invariant functional \eqref{wrong} is nonincreasing for $r\in (0,\frac 13)$ for any $0<\alpha<Q$. Instead, the right-invariant functional in \eqref{D2} above,
\[
r\ \longrightarrow\ \mathscr D_2(f_+,f_-,r) = \frac{1}{r^4} \mathscr D_2(f_+,r) \mathscr D_2(f_-,r),
\]
is nondecreasing on $(0,\infty)$.
\end{proposition}

This note contains four sections. Besides the present one, in Section \ref{S:back} we collect some background material that is needed in the rest of the paper. In Section \ref{S:proof} we prove Theorems \ref{T:babyACF}, \ref{T:babyC} and Proposition \ref{P:ce}, and discuss the role that Bochner formulas play in these results. In Section \ref{S:almgren} we discuss another famous monotonicity formula, that of Almgren \cite{A}, and we show that, in accordance with the results in \cite{GL, GR}, its sub-Riemannian counterpart generically fails. However, the fundamental question of whether or not the frequency \eqref{fre} be locally bounded, remains open at the moment.

In closing, we hope that the present note helps to clarify some of the critical aspects connected to monotonicity in non-Riemannian ambients and at the same time provides an incentive for further understanding.


\section{Background material}\label{S:back}

In this section we collect some background material that is needed in the rest of the paper. To keep the preliminaries at a minimum and avoid pointless repetitions, we routinely use from now on the definitions and notations from the paper \cite{GR}, where some Almgren type monotonicity formulas in Carnot groups and for Baouendi-Grushin operators were obtained (for the latter, see also the first papers on the subject \cite{GL, Gjde}). 
 A Carnot group of step $k\ge 1$ is a simply-connected real Lie group $(\bG, \circ)$ whose Lie algebra $\bg$ is stratified and $k$-nilpotent. This means that there exist vector spaces $\bg_1,...,\bg_k$ such that:  
\begin{itemize}
\item[(i)] $\bg=\bg_1\oplus \dots\oplus\bg_k$;
\item[(ii)] $[\bg_1,\bg_j] = \bg_{j+1}$, $j=1,...,k-1,\ \ \ [\bg_1,\bg_k] = \{0\}$.
\end{itemize}
We assume that $\bg$ is endowed with a
scalar product $\langle\cdot,\cdot\rangle$ with respect to which the layers $\bg_j's$, $j=1,...,r$,
are mutually orthogonal. We let $m_j =$ dim$\ \bg_j$, $j=
1,...,k$, and denote by $N = m_1 + ... + m_k$ the topological
dimension of $\bG$. From the assumption (ii) on the Lie
algebra it is clear that any basis of the first layer $\bg_1$ bracket generates the whole of $\bg$. Because of such special role $\bg_1$ is usually called the horizontal
layer of the stratification. For ease of notation we henceforth write $m = m_1$. In the case in which $k =1$ one has $\bg = \bg_1$, and thus $\bG$ is isomorphic to $\R^m$. There is no sub-Riemannian geometry involved and everything is classical. 
We are primarily interested in the genuinely non-Riemannian setting $k>1$. 

Henceforth, given a horizontal Laplacian $\sul$ as in \eqref{L} above, we indicate with $\Gamma(g,g') = \Gamma(g',g)$ the unique positive
fundamental solution of $-\Delta_H$ which goes to zero at infinity. Such distribution is left-translation invariant, i.e., one has
\[
\Gamma(g,g') = \tilde \Gamma(g^{-1}\circ g'),
\]
for some function $\tilde \Gamma \in C^\infty(\bG\setminus\{e\})$, where $e\in \bG$ is the group identity.
For every $r>0$, let 
\begin{equation}\label{balls}
B_r = \left\{g\in \bG \mid \Gamma(g,e)>\frac{1}{r^{Q-2}}\right\}.
\end{equation}
It was proved by Folland in \cite{Fo} that the distribution $\tilde \Gamma(g)$ is homogeneous of degree $2-Q$ with respect to the non-isotropic dilations in $\bG$ associated with the stratification of its Lie algebra $\bg$. This implies that, if we define
\begin{equation}\label{rho}
\rho(g) = \tilde \Gamma(g)^{- 1/(Q-2)},
\end{equation}
then the function $\rho$ is homogeneous of degree one. Notice that $\rho\in C^\infty(\bG\setminus\{e\}) \cap C(\bG)$. We obviously have from \eqref{balls}
\begin{equation}\label{br}
B_r = \{g\in \bG\mid \rho(g)<r\}.
\end{equation}
Henceforth, we will use the notation $S_r = \p B_r$. 

Next, denote by $p(g,g',t)$ the positive and symmetric heat kernel for $\sul - \p_t$ constructed by Folland in \cite{Fo}. We recall the following result, which combines \cite[Theor. IV.4.2 and Theor. IV.4.3]{VSC}. In what follows, if $\ell\in \mathbb N\cup\{0\}$, we consider multi-indices $(j_1,...,j_\ell)$, with $j_1,...,j_\ell\in\{1,...,m\}$.

\begin{theorem}\label{T:VSC}
There exists $C, C'>0$ such that for all $g, g'\in \bG$ and $t>0$ one has
\[
p(g,g',t) \ge \frac{C}{t^{\frac Q2}} e^{- C' \frac{d(g,g')^2}t}.
\]
Furthermore, for every $s, \ell\in \mathbb N\cup\{0\}$ and $\ve>0$, there exists $C>0$ such that for all $g, g'\in \bG$ and $t>0$ one has
\[
\left|\p_t^s X_{j_1}X_{j_2}...X_{j_\ell} p(g,g',t)\right| \le \frac{C}{t^{\frac Q2 + s + \frac{\ell}2}}  e^{-\frac{d(g,g')^2}{4(1+\ve)t}}.
\]
\end{theorem}
The heat semigroup $P_t = e^{-t \sul}$ is defined on a reasonable function $f:\bG\to \R$ by the formula
\[
P_t f(g) = \int_{\bG} p(g,g',t) f(g') dg'.
\]
Similarly to the classical case, the function $u(g,t) = P_t f(g)$ is smooth in $\bG\times (0,\infty)$ and solves the Cauchy problem 
\[
\sul u - p_t u = 0\ \ \ \ \ \text{in}\ \bG\times (0,\infty),\ \ \ \ \ u(g,0) = f(g),\ \ \ \ g\in \bG.
\]
If we assume that there exist $A, \alpha>0$ such that for every $g\in \bG$ one has 
\begin{equation}\label{gaussian0}
|f(g)| \le  A\ e^{\alpha d(g,e)^2},
\end{equation} 
where we have denoted by $d(g,g')$ the control distance in $\bG$ associated with the horizontal layer $\bg_1$ of the Lie algebra, then the semigroup $P_t f(g)$ is well-defined, at least for $0<t<T$, where $T = T(\alpha)>0$ is sufficiently small. For this it suffices to observe that, if $T < \frac{1}{4(1+\ve)\alpha}$, then for $0<t<T$ one has for any $g\in \bG$
\[
|P_t f(g)| \le \int_{\bG} |f(g')| p(g,g',t) dg' \le C A e^{2\alpha d(g,e)^2} \int_{\bG} e^{-d(g',g)^2\left[\frac{1}{4(1+\ve)T} - \alpha\right]} dg' < \infty.
\] 

For $r>0$ consider now the parabolic cylinders
\[
Q_r = B_r \times (- r^2,0).
\]
As a special case of \cite[Theor. 1.1]{DG} we obtain the following.

\begin{theorem}\label{T:PS}
Suppose that $f$
solves $\sul f - \p_t f= c$ in $\bG\times \R$, for some $c\in \R$. For every $s, \ell \in
\mathbb N \cup \{0\}$ and $r>0$, one has 
\[
\underset{Q_{r/2}}{\sup} \bigg|\p_t^s
X_{j_1}X_{j_2}...X_{j_\ell}  f\bigg| \leq \frac{C}{r^{2s + \ell}}
\frac{1}{|Q_{2r}|} \int_{Q_{2r}} |f| dg' d\tau,
\]
for some constant $C=C(c,s,\ell)>0$. 
\end{theorem}


\section{Proof of Theorems \ref{T:babyACF}, \ref{T:babyC} and Proposition \ref{P:ce}}\label{S:proof}

In this section we prove Theorems \ref{T:babyACF} and \ref{T:babyC}, as well as Proposition \ref{P:ce}. 
With these preliminaries in place, we now return to the functional \eqref{JG} and observe that, since the function $g\to \rho(g)$ is homogeneous of degree one  with respect to the nonisotropic group dilations $\{\delta_\la\}_{\la>0}$, while $g\to |\nabla_H\rho(g)|^2$ is homogeneous of degree zero with respect to the same, the change of variable $g' = \delta_r(g)$, for which $dg' = r^Q dg$, immediately gives 
\[
\frac{1}{r^\alpha} \int_{B_r} \frac{1}{\rho^{Q-\alpha}} |\nabla_H\rho(g)|^2 dg =  \int_{B_1} \frac{1}{\rho^{Q-\alpha}} |\nabla_H\rho(g)|^2 dg = \omega_\alpha>0.
\]
This proves \eqref{omega}. The statement \eqref{zero} immediately follows from the continuity of $f$ and from \eqref{omega}.

Next, we record the following equation (see \cite[formula (3.12)]{GR} or also the earlier work \cite{CGL} for a more general result), valid for any function $\psi\in C^2(\bG)$,
\begin{equation}\label{2psiG}
\psi(e) = \frac{Q-2}{r^{Q-1}} \int_{S_r} \psi(g)
\frac{|\nabla_H\rho(g)|^2}{|\nabla\rho(g)|} dH_{N-1}(g) -
\int_{B_r} \Delta_H \psi(g) \big[\frac{1}{\rho^{Q-2}}
-\frac{1}{r^{Q-2}}\big] dg.
\end{equation}
The equation \eqref{2psiG} represents a generalisation of Gaveau's mean value formula in \cite{Ga} for harmonic functions in the Heisenberg group $\Hn$. Differentiating with respect to $r$ in \eqref{2psiG} we obtain
\begin{equation}\label{2usqG}
\frac{d}{dr} \frac{1}{r^{Q-1}} \int_{S_r} \psi(g)
\frac{|\nabla_H\rho(g)|^2}{|\nabla\rho(g)|} dH_{N-1}(g) =
\frac{1}{r^{Q-1}} \int_{B_r} \sul \psi(g) dg.
\end{equation}
From \eqref{2usqG} we immediately infer the following result.

\begin{lemma}\label{L:sbm}
Suppose that $\psi\in C^2(B_1)$. If $\sul \psi \ge 0$ ($\le 0$) in $B_1$ then the averages 
\[
r\to \frac{1}{r^{Q-1}} \int_{S_r} \psi(g)
\frac{|\nabla_H\rho(g)|^2}{|\nabla\rho(g)|} dH_{N-1}(g)
\]
are nondecreasing (nonincreasing) in $r\in (0,1)$.
\end{lemma}

Returning to the functional $\mathscr D_\alpha(f,r)$, we have the following simple, yet important, fact. 

\begin{proposition}\label{P:mono}
Suppose that the surface averages of $f$,
\begin{equation}\label{ave}
r\to \frac{1}{r^{Q-1}} \int_{S_r} f(g)
\frac{|\nabla_H\rho(g)|^2}{|\nabla\rho(g)|} dH_{N-1}(g),
\end{equation}
are nondecreasing (nonincreasing) in $r\in (0,1)$. Then $r\to \mathscr D_\alpha(f,r)$ is nondecreasing (nonincreasing) in $(0,1)$ and we have for every $r\in (0,1)$
\begin{equation}\label{e}
\omega_\alpha f(e) \le \mathscr D_\alpha(f,r),
\end{equation}
where $\omega_\alpha>0$ is the universal constant in \eqref{omega}. 
\end{proposition}

\begin{proof}
Using Federer's coarea formula to differentiate \eqref{JG} one has
\[
\mathscr D_\alpha'(f,r) = -\frac{\alpha}{r^{\alpha +1}} \int_{B_r} \frac{f(g)}{\rho^{Q-\alpha}} |\nabla_H\rho(g)|^2 dg + \frac{1}{r^Q} \int_{S_r} f(g) \frac{|\nabla_H\rho(g)|^2}{|\nabla \rho(g)|} d\sigma.
\]
Assume that \eqref{ave}
are nondecreasing in $r\in (0,1)$. Again the coarea formula gives
\begin{align*}
& \frac{\alpha}{r^{\alpha +1}} \int_{B_r} \frac{f(g)}{\rho^{Q-\alpha}} |\nabla_H\rho(g)|^2 dg = \frac{\alpha}{r^{\alpha +1}} \int_0^r \int_{S_t} \frac{f(g)}{\rho^{Q-\alpha}} \frac{|\nabla_H\rho(g)|^2}{|\nabla \rho(g)|} d\sigma dt
\\
& = \frac{\alpha}{r^{\alpha +1}} \int_0^r t^{\alpha -1} \frac{1}{t^{Q-1}}\int_{S_t} f(g) \frac{|\nabla_H\rho(g)|^2}{|\nabla \rho(g)|} d\sigma dt
\\
& \le \frac{\alpha}{r^{\alpha +1}} \frac{1}{r^{Q-1}}\int_{S_r} f(g) \frac{|\nabla_H\rho(g)|^2}{|\nabla \rho(g)|} d\sigma \int_0^r t^{\alpha -1}  dt
\\
& = \frac{1}{r^Q} \int_{S_r} f(g) \frac{|\nabla_H\rho(g)|^2}{|\nabla \rho(g)|} d\sigma.
\end{align*}
This proves that $\mathscr D_\alpha'(r)\ge 0$ for $r\in (0,1)$. Similarly, one proves that $\mathscr D_\alpha'(r)\le 0$ if \eqref{ave} are nonincreasing. The second part of Proposition \ref{P:mono} is a direct consequence of the first, and of \eqref{zero}. 

\end{proof}

\begin{remark}\label{R:sub}
Since in view of Lemma \ref{L:sbm} the monotonicity of \eqref{ave} characterises sub- and superharmonicity, a similar monotonicity holds true for $r\to \mathscr D_\alpha(f,r)$ if $f$ is sub- or superharmonic in $B_1$. 
\end{remark}

We next recall that the celebrated identity of Bochner states that on a Riemannian manifold $M$ one has for $f\in C^3(M)$
\begin{equation}\label{boe}
\Delta(|\nabla f|^2) = 2 ||\nabla^2 f||^2 + 2 \langle\nabla(\Delta f),\nabla f\rangle + 2 \operatorname{Ric}(\nabla f,\nabla f),
\end{equation}
where $\operatorname{Ric}(\cdot,\cdot)$ indicates the Ricci tensor on $M$, see e.g. \cite[Sec. 4.3 on p.18]{CN}. This implies in particular that if $\Delta f = c$ for some $c\in \R$, and $\operatorname{Ric}(\cdot,\cdot)\ge 0$, then
\begin{equation}\label{bosub}
\Delta(|\nabla f|^2) \ge 2 ||\nabla^2 f||^2 \ge 0.
\end{equation}
As we will see in a short while, in sub-Riemannian geometry the fundamental subharmonicity property \eqref{bosub} fails miserably. This negative situation can be remedied by bringing the right-invariant vector fields $\tilde X_i$ to center stage. As we have mentioned, in free boundary problems the idea of working with right-invariant derivatives was first systematically developed in \cite{DGP} to establish the $C^{1,\alpha}$ regularity of the free boundary in the non-holonomic obstacle problem. A related perspective was further exploited in \cite{Gmanu2} to prove $C^{1,\alpha}$ regularity via maximum principles, and subsequently in the study of fully nonlinear equations in \cite{MM}, and of sub-Riemannian mean curvature flow in \cite{CCM}.

\begin{proposition}[Right Bochner type identity]\label{P:Br}
Let $\bG$ be a Carnot group, $f\in C^3(\bG)$,
then one has
\begin{align}\label{B1r}
\sul(|\nhh f|^2) = 2 \langle\nhh f,\nhh(\sul f)\rangle + 2 \sum_{i=1}^m |\tilde{\nabla}_H(X_i f)|^2.
\end{align}
If in particular $\sul f = c$, for some $c\in \R$, then we have
\begin{equation}\label{B2r}
\sul(|\nhh f|^2) = 2 \sum_{i=1}^m |\tilde{\nabla}_H(X_i f)|^2 \ge 0.
\end{equation}
\end{proposition}

\begin{proof}
The proof is a straightforward calculation that uses the commutation identities $[X_i,\tilde X_j] = 0$, $i, j =1,...,m$. We leave the details to the interested reader. 

\end{proof}

We emphasise that the two objects $|\nhh f|^2$ and $|\nh f|^2$ differ substantially. For instance, in the special case in which $\bG$ is a group of step $k=2$, with group constants $b^{\ell}_{ij}$, and (logarithmic) coordinates $g = (z_1,...,z_m,\sigma_1,...,\sigma_{m_2})$, one has
\begin{equation}\label{difference}
|\nh f|^2- |\nhh f|^2 =  2 \sum_{\ell=1}^{m_2} \left(\sum_{1\le i<j\le m} b^{\ell}_{ij} \left(z_i \p_{z_j} f - z_j \p_{z_{i} }f\right)
\right) \p_{\sigma_{\ell}} f,
\end{equation}
see \cite[Lemma 2.3]{Gmanu2}.

We can now present the
\begin{proof}[Proof of Theorem \ref{T:babyACF}]
Suppose $\sul f = c$ in $B_1$. By hypoellipticity, we know that $f\in C^\infty(B_1)$. At this point the desired conclusion is an immediate consequence of Proposition \ref{P:Br}, Lemma \ref{L:sbm} and Proposition \ref{P:mono}.

\end{proof}

Next we present the
\begin{proof}[Proof of Theorem \ref{T:babyC}]
Let $f$ be a solution of $\p_t f - \sul f = c$ in the infinite slab $\bG\times (-1,0)$. By the hypoellipticity result in \cite{Ho}, we know that $f\in C^\infty(\bG\times(-1,0))$. However, now we cannot proceed as in the proof of Theorem \ref{T:babyACF} since the set of integration is not a relatively compact set (the pseudoballs $B_r$). To make sense of the integral in \eqref{babyC} on a sufficiently small interval $t\in (0,T)$ and be able to differentiate it with respect to the parameter $t\in (-1,0)$, we use the assumption \eqref{gaussian}. Note that we can write \eqref{babyC} as follows 
\begin{equation}\label{H}
\mathscr I(|\nhh f|^2,t) = \frac 1t \int_0^t P_\tau(|\nhh f(\cdot,-\tau)|^2)(e) d\tau,
\end{equation}
provided that the function $u(g,t) = |\nhh f(g,-t)|^2)$ is such that the integral defining \[
P_t(|\nhh f(\cdot,-t)|^2)(e) = \int_{\bG} p(g,e,t) |\nhh f(g,-t)|^2 dg
\]
be finite. From Theorem \ref{T:PS} we now have for every $\ell \in \mathbb N$ and $r>0$
\begin{equation}\label{same}
\underset{Q_{r/2}}{\sup} \bigg|
X_{j_1}X_{j_2}...X_{j_\ell}  f\bigg| \leq \frac{C}{r^{\ell}}
\frac{1}{|Q_{2r}|} \int_{Q_{2r}} |f(g',\tau)| dg' d\tau \le 
\frac{A C}{r^{\ell}}
\frac{1}{|B_{2r}|} \int_{B_{2r}} e^{\alpha d(g',e)^2} dg',
\end{equation}
where in the last inequality we have used \eqref{gaussian} and the fact that $|Q_{2r}| = 4r^2 |B_{2r}|$. From \eqref{same} it is easy to show that $X_{j_1}X_{j_2}...X_{j_\ell}  f$ satisfies the same uniform estimate in \eqref{gaussian} as $f$. Since any right-invariant derivative $\tilde X_j f$ can be expressed in terms of the vector fields $X_j$ and a certain number of combinations, with polynomial coefficients, of terms $X_{j_1}X_{j_2}...X_{j_\ell}  f$, by \eqref{same} we obtain a similar a priori estimate for $|\nhh f|^2$, possibly with a larger coefficient $\alpha>0$ in the exponential. This implies that $P_\tau(|\nhh f(\cdot,-\tau)|^2)(e)$ is well-defined for $0<\tau<T$, for some $T= T(\alpha)>0$ (see the discussion prior to Theorem \ref{T:PS}).
Differentiating \eqref{H} we thus find for every $t\in (0,T)$
\begin{align*}
\frac{d}{dt} \mathscr I(|\nhh f|^2,t) = - \frac 1t \mathscr I(|\nhh f|^2,t) + 
\frac 1t P_t(|\nhh f(\cdot,-t)|^2)(e).  
\end{align*}
We infer that $t \longrightarrow \mathscr I(|\nhh f|^2,t)$ is nondecreasing (nonincreasing) in $(0,T)$ if and only if we have for every $t\in (0,T)$
\begin{equation}\label{eq}
\mathscr I(|\nhh f|^2,t)\ \le\ (\ge)\ P_t(|\nhh f(\cdot,-t)|^2)(e).
\end{equation}
We next differentiate the functional in the right-hand side of \eqref{eq} obtaining by the chain rule
\begin{align}\label{momo}
& \frac{d}{dt} \left\{P_t(|\nhh f(\cdot,-t)|^2(e))\right\} = P_t(\frac{d}{dt} (|\nhh f(\cdot,-t)|^2))(e) + \frac{dP_t}{dt}(|\nhh f(\cdot,-t)|^2)(e)
\\
& = - 2 P_t(\langle \nhh f(\cdot,t),\nhh(\p_t f(\cdot,-t))\rangle)(e) + \sul P_t(|\nhh f(\cdot,-t)|^2)(e)
\notag\\
& = - 2 P_t(\langle \nhh f(\cdot,-t),\nhh(\p_t f(\cdot,-t))\rangle)(e) + P_t(\sul(|\nhh f(\cdot,-t)|^2))(e)
\notag\\
& = 2 P_t(\langle\nhh f(\cdot,-t),\nhh(\sul f-\p_t f))(\cdot,-t)\rangle)(e)
\notag\\
& + 2 \sum_{i=1}^m P_t(|\tilde{\nabla}_H(X_i f)(\cdot,-t)|^2)(e),
\notag
\end{align}
where in the last equality in \eqref{momo} we have used \eqref{B1r} in Proposition \ref{P:Br}. Since we are assuming that $\sul f-\p_t f = c$ in $\bG\times (-1,0)$, we infer from \eqref{momo}
\begin{equation}\label{momo2}
\frac{d}{dt} \left\{P_t(|\nhh f(\cdot,-t)|^2(e))\right\} = 2 \sum_{i=1}^m P_t(|\tilde{\nabla}_H(X_i f)(\cdot,-t)|^2)(e)\ \ge\ 0,
\end{equation}
therefore the functional $t \longrightarrow P_t(|\nhh f(\cdot,-t)|^2)(e)$ is nondecreasing. This implies
\[
\mathscr I(|\nhh f|^2,t) = \frac 1t \int_0^t P_\tau(|\nhh f(\cdot,-\tau)|^2)(e) d\tau \le P_t(|\nhh f(\cdot,-t)|^2)(e),
\]
which finally proves \eqref{eq}, and therefore the nondecreasing monotonicity of $t\longrightarrow \mathscr I(|\nhh f|^2,t)$. 

\end{proof}

Having established the positive results, we next discuss the typically non-Riemannian phenomenon for which Theorems \ref{T:babyACF} and \ref{T:babyC} fail if in their statement one replaces the right-invariant \emph{carr\'e du champ} with the left-invariant one $|\nabla_H f|^2$.
We recall the following result which is \cite[Proposition 3.3]{Gmanu}.  

\begin{proposition}[Left Bochner type identity]\label{P:B}
Let $\bG$ be a Carnot group, $f\in C^3(\bG)$,
then one has
\begin{align}\label{B1}
\sul (|\nh f|^2) & = 2 ||\nabla_H^2 f||^2 + 2 \langle\nh f , \nh(\sul f)\rangle 
 + \frac{1}{2} \sij ([X_i,X_j]f)^2
\\
&  + 4 \sij X_j f [X_i,X_j] X_i f + 2 \sij X_j f\ [X_i,[X_i,X_j]]
f. 
\notag
\end{align}
\end{proposition} 
In \eqref{B1} we have denoted by $\nabla_H^2 f = [f_{ij}]$ the symmetrised horizontal Hessian of $f$ with entries
\[
f_{ij} = \frac{X_i X_j f + X_j X_i f}{2}.
\]
When $\bG$ is of step $2$, then $[X_i,[X_i,X_j]] = 0$ and we obtain
from Proposition \ref{P:B}.

\begin{corollary}\label{C:B}
Let $\bG$ be a Carnot group of step $k=2$, $f\in C^3(\bG)$,
then one has
\begin{align}\label{B1}
\sul (|\nh f|^2) & = 2 ||\nabla_H^2 f||^2 + 2 \langle\nh f , \nh(\sul u)\rangle 
 + \frac{1}{2} \sij ([X_i,X_j]f)^2
\\
&  + 4 \sij X_j f [X_i,X_j] X_i f. \notag
\end{align}
\end{corollary}
The problem with \eqref{B1} is that, even if $\sul f = 0$, the term $4 \sij X_j f [X_i,X_j] X_i f$ can prevail so badly on the positive terms, to reverse the sign of the sum in the right-hand side. 
We have already hinted to this phenomenon with the example \eqref{fmia}, see \eqref{ted}. For the reader's understanding, we next discuss this aspect in more detail. Consider the Heisenberg group $\bG = \Hn$ with the left-invariant basis of the Lie algebra given by 
\begin{equation}\label{vf}
X_i = \p_{x_i} - \frac{y_i}2\ \p_\sigma,\ \ \ \ \ X_{n+i} = \p_{y_i} + \frac{x_i}2\ \p_\sigma,\ i=1,...,n.
\end{equation}
If we let $T = \p_\sigma$, then the only nontrivial commutators are $[X_i,X_{n+j}] = T\  \delta_{ij}$, and we find 
\[
\sij ([X_i,X_j]u)^2 = \sum_{i,j=1}^{2n} ([X_i,X_j]u)^2 = 2 \sum_{i<j} ([X_i,X_j]u)^2 = 2n (Tu)^2.
\]
Similarly, we have
\[
\sij X_j u [X_i,X_j] X_iu = \sum_{i<j}  X_j u [X_i,X_j] X_iu - \sum_{i<j}  X_i u [X_i,X_j] X_ju = \langle\nh(Tu),\nh^\perp u\rangle,
\]
where we have denoted by $\nh^\perp u = (X_{n+1}u,...,X_{2n}u, - X_1 u,...,-X_n u)$. 
Substituting the latter two equations in \eqref{B1} we obtain 
\begin{align}\label{B2}
\sul (|\nh f|^2) & = 2 ||\nabla_H^2 f||^2 + 2 \langle\nh f , \nh(\sul f)\rangle +
 +  n (T f)^2
\\
&  + 4 \langle\nh(Tf),\nh^\perp f\rangle. 
\notag
\end{align}
Now, if $\sul f = c$, with $c\in \R$, then one has from \eqref{B2} 
\begin{align}\label{B22}
\sul (|\nh f|^2)  = 
2 ||\nabla_H^2 f||^2 +  n (T f)^2 + 4 \langle\nh(Tf),\nh^\perp f\rangle.
\end{align}
The following discussion shows that the term $4 \langle\nh(Tf),\nh^\perp f\rangle$ can destroy the subharmonicity of $|\nh f|^2$. Consider the harmonic function \eqref{f} from the work \cite[Sec.5]{FF3}, but \eqref{fmia} would work equally well. Such function is the sum of two solid harmonics of degree one and three. Greiner first computed such solid harmonics in $\mathbb H^1$, see \cite[p. 387]{Gr}, and Dunkl subsequently generalised his results to $\Hn$ in \cite{Du}. The subject has since somewhat languished for lack of a complete understanding of some fundamental orthogonality and completeness issues, see the unpublished preprint \cite[p.29]{GK}, but also the discussion in Section \ref{S:almgren}. 

\begin{proof}[Proof of Proposition \ref{P:ce}]
Instead of the lengthy calculations based on spherical harmonics in \cite[Sec. 4, 5]{FF3}, we disprove the nondecreasing monotonicity of the left-invariant functional
\begin{equation}\label{wrong}
r\ \longrightarrow\ \frac{1}{r^\alpha} \int_{B_r} \frac{|\nh f(g)|^2}{\rho(g)^{Q-\alpha}} |\nabla_H\rho(g)|^2 dg
\end{equation}
by simply observing that, on the function \eqref{f}, we have $\sul(|\nh f|^2)\le 0$ in an infinite cylinder in $\mathbb H^1$. We then use Lemma \ref{L:sbm} and Proposition \ref{P:mono} to deduce the nonincreasing monotonicity of \eqref{wrong}. From \eqref{f} and \eqref{vf} simple computations give 
\begin{equation}\label{Xu}
X_1 f = 1 - 3 |z|^2,\ \ \ \ \ \ \ \ \ \ X_2 f = 6\sigma + 3 x y,
\end{equation}
and furthermore
\begin{equation}\label{X2u}
X_1^2 f = - 6 x,\ \ \ \ \ \ \ \ \ \ X_2^2 f = 6x.
\end{equation}
In particular $\sul f = 0$ in $\Ho$ (this conclusion is also obvious from the fact that $f$ is the sum of two harmonic polynomials). Using \eqref{Xu} we now find
\begin{equation}\label{nh}
|\nh f|^2 = 1 + 9 |z|^4 - 6 |z|^2 + 36 \sigma^2 + 9 x^2 y^2 + 36 xy\sigma.
\end{equation}
We next prove that, contrarily to the Riemannian case \eqref{boe}, the function $|\nh f|^2$ badly fails to be subharmonic. We compute from \eqref{nh}
\begin{align*}
X_1(|\nh f|^2) & = 36 x |z|^2 - 12 x + 18 xy^2 + 36 y\sigma - 36 y\sigma - 18 xy^2
\\
& = 36 x |z|^2 - 12 x,
\end{align*}
and 
\begin{align*}
X_2(|\nh f|^2) & = 36 y |z|^2 - 12 y + 18 x^2 y + 36 x\sigma + 36 x\sigma + 18 x^2 y
\\
& = 36 y |z|^2 - 12 y + 36 x^2 y + 72 x\sigma.
\end{align*}
Next,
\begin{align*}
X_1^2(|\nh f|^2) & = 36  |z|^2 + 72 x^2 - 12,
\end{align*}
and 
\begin{align*}
X_2^2(|\nh f|^2) & = 36 |z|^2 + 72 y^2 - 12  + 72 x^2
\\
& = 108 |z|^2 - 12.
\end{align*}
Combining the latter two equations we find
\begin{equation}\label{fboc}
\sul (|\nh f|^2)= 216 x^2 + 144 y^2 - 24.
\end{equation}
It is now clear from \eqref{fboc} that
\begin{equation}\label{bad}
\sul (|\nh f|^2) \le 216 |z|^2 - 24 \le 0,
\end{equation}
provided that $|z|^2 \le \frac 19$. From Lemma \ref{L:sbm} and Proposition \ref{P:mono} we conclude that for the harmonic function $f$ in \eqref{f} the functional
\[
r\ \longrightarrow\ \mathscr D_2(|\nh f|^2,r)
\]
is nonincreasing for $r\in (0,1/3)$! 

For the second part of the proposition we need to compute $|\nhh f|^2$. We have
\[
\tilde X_1 f = f_x + \frac y2 f_\sigma = 1 - 3x^2 + 3 y^2,\ \ \ \tilde X_2 f = f_y - \frac x2 f_\sigma = 6 \sigma - 3 x y,
\]
and therefore
\begin{equation}\label{gud}
|\nhh f|^2 = (1 - 3x^2 + 3 y^2)^2 + (6 \sigma - 3 x y)^2.
\end{equation}
By \eqref{gud}, the fact that $|\nh \rho|^2 = \frac{|z|^2}{\rho^2}$, and the change of variable $(x,y,\sigma)\to (-x,-y,\sigma)$ (see \cite[formula (6.2)]{FF3}), we easily recognise that
\[
\mathscr D_\alpha(f_+,r) = \mathscr D_\alpha(f_-,r).
\]
Therefore, thanks to \eqref{B2r} in Proposition \ref{P:Br} and our Theorem \ref{T:babyACF}, we know that
\[
r\ \longrightarrow\ \mathscr D_\alpha(f_+,r) = \frac 12 \mathscr D_\alpha(f,r)\ \text{is nondecreasing for}\ r\in (0,\infty).
\] 
As a consequence, we infer that $r \longrightarrow \mathscr D_2(f_+,f_-,r) = \frac 14 \mathscr D_2(f,r)^2$ 
is nondecreasing on $(0,\infty)$.

\end{proof}

\begin{remark}\label{R:due}
It is interesting to observe that with $f$ as in \eqref{f} we have instead in the entire space $\mathbb H^1$
\[
\sul(|\nh f|^2 + \frac 13 (Tf)^2) =  216 x^2 + 144 y^2 - 24 + 24 \ge 0.
\]
As a consequence, the functional
$r\ \longrightarrow\ \mathscr D_2(|\nh f|^2+ \frac 13 (Tf)^2,r)$
is globally nondecreasing.
\end{remark}


\section{Failure of Almgren monotonicity formula in sub-Riemannian geometry}\label{S:almgren}

In this final section we disscuss the sub-Riemannian counterpart of another celebrated monotonicity formula from geometric PDEs. We recall that, in its simplest form, Almgren monotonicity formula states that if $\Delta f = 0$ in $B_1\subset \Rn$, then its \emph{frequency}
\[
N(f,r) = \frac{r \int_{B_r} |\nabla f|^2 dx}{\int_{S_r} f^2 d\sigma}
\]
is nondecreasing, see \cite{A}. This result plays a fundamental role in several areas of analysis and geometry, ranging from minimal surfaces, to unique continuation for elliptic and parabolic PDEs, and more recently free boundaries in which the obstacle is confined to a lower-dimensional manifold. We refer in particular to the papers \cite{GL1, GL2}, and to the more recent works \cite{ACS, CSS, GP, DGPT, BG}. 

In sub-Riemannian geometry the horizontal Laplacian \eqref{L} is not real-analytic hypoelliptic in general, and a fundamental open question is whether harmonic functions have the unique continuation property (ucp). An initial very interesting study of what can go wrong for smooth, even compactly supported, perturbations of \eqref{L} was done by H. Bahouri in \cite{Ba}. However, Bahouri's work does not provide any evidence, in favour or to the contrary, about the ucp for harmonic functions in a Carnot group. The reader is referred to \cite{GR} for a detailed discussion. In the same paper the authors have shown that, in a Carnot group $\bG$, given a harmonic function $f$ in a ball $B_1 \subset \bG$, the following sub-Riemannian analog of Almgren frequency
\begin{equation}\label{fre} 
N(f,r) = \frac{r \int_{B_r} |\nh f|^2 dg}{\int_{S_r} f^2 |\nh \rho| d\sigma_H}
\end{equation}
is nondecreasing in $r\in (0,1)$ provided that $f$ has vanishing \emph{discrepancy}, see also \cite{GL} for the first result in this direction in $\Hn$. In the surface integral in \eqref{fre} the symbol $d\sigma_H$ denotes the horizontal perimeter measure. It is obvious that if the frequency is nondecreasing on an interval $(0,r_0)$, then one has in particular $N(f,\cdot)\in L^\infty(0,r_0)$. In \cite[Theor. 4.3]{GR} it was shown that, in fact, the local boundedness of $N(f,\cdot)$ is necessary and sufficient for the following \emph{doubling condition}
\begin{equation}\label{dc}
\int_{B_{2r}} f^2 dg \le C \int_{B_r} f^2 dg, \ \ \ \ \ 0<r<r_0.
\end{equation}
It is well-known by now (see \cite{GL1}) that \eqref{dc} implies the strong unique continuation property for $f$.

In a Carnot group $\bG$ the local boundedness of the frequency of a harmonic function $f$ is a fundamental open problem (to be proved, or disproved). In \cite[Theor.8.1]{GR} it was shown that \eqref{dc} is true for harmonic functions in a Metivier group, and therefore in such Lie groups (which include those of Heisenberg type) the frequency \eqref{fre} is locally bounded. The following discussion shows that not even in $\Hn$ one should expect the frequency to be generically nondecreasing. We emphasise that this phenomenon of monotonicity versus boundedness is connected to the ``almost monotonicity" character of the conjecture in Section \ref{S:proof}.

We recall that in \cite[Prop.3.6]{GR} it was shown that if $f$ is harmonic in a Carnot group, then
\begin{equation}\label{dir}
\int_{B_r} |\nh f|^2 dg = \frac 1r \int_{S_r} f Zf |\nh \rho| d\sigma_H,
\end{equation}
where $Z$ denotes the generator of the group dilations in $\bG$. Combining \eqref{fre} with \eqref{dir} we see that we can express the frequency in the useful alternative fashion
\begin{equation}\label{fre2}
N(f,r) = \frac{\int_{S_r} f Zf |\nh \rho| d\sigma_H}{\int_{S_r} f^2 |\nh \rho| d\sigma_H}.
\end{equation}
We emphasise that \eqref{fre2} does immediately imply that if $f$ is a harmonic function homogeneous of degree $\kappa$, then $N(f,r) \equiv \kappa$. We do not know whether the opposite implication holds in general! The main reason is that, even when $\bG = \Rn$, the only known proof of such implication seem to crucially rest on the full-strength of Almgren monotonicity formula.

Suppose now that $P_h$ and $P_k$ are two harmonic functions in $\bG$, respectively of homogeneous degree $h\not= 0$ and $k\not= 0$, and suppose to fix the ideas that $h<k$. If $f = P_h + P_k$, we have
\[
f Zf = f\left(ZP_h + Z P_k\right) = f\left(h P_h + k P_k\right) = h f^2 + (k-h) f P_k.
\]
Inserting this information in \eqref{fre2} we find
\begin{equation}\label{fre3}
N(f,r) = h + (k-h) \frac{\int_{S_r} f P_k |\nh \rho| d\sigma_H}{\int_{S_r} f^2 |\nh \rho| d\sigma_H}.
\end{equation}
It is clear from \eqref{fre3} that on a harmonic function of the type $f = P_h + P_k$ the frequency is nondecreasing if and only if such is the quantity 
\[
\mathscr E(r) = \frac{\int_{S_r} f P_k |\nh \rho| d\sigma_H}{\int_{S_r} f^2 |\nh \rho| d\sigma_H}.
\]
Suppose that, similarly to the case $\bG = \Rn$, we knew
\begin{equation}\label{wrong2}
\int_{S_1} P_h P_k |\nh \rho| d\sigma_H = \begin{cases} 0,\ \ \ \ \ \text{if}\ h\not= k,
\\
a_h >0\ \ \ \ \ \ \text{if}\ h = k.
\end{cases}
\end{equation}
From \eqref{wrong2} we would immediately infer by rescaling ($d\sigma_H \circ \delta_r = r^{Q-1} d\sigma_H$) that
\[
\mathscr E(r) = \frac{a_k r^{k-h}}{a_h + a_k r^{k-h}},
\]
and this would easily imply $\mathscr E'(r) \ge 0$. But in sub-Riemannian geometry the ``Euclidean" looking identity \eqref{wrong2} fails to be true in general. This negative phenomenon was already brought to light in the context of $\Hn$ in \cite[Theor. 1.1]{GL}, and this is why that result contained the additional assumption (1.19), and  in \cite[Def. 5.1]{GR} the notion of \emph{discrepancy} was introduced. What is true, instead, in any Carnot group, is the following formula
\begin{equation}\label{true}
\int_{S_r} P_h \frac{\langle \nh P_k,\nh \rho\rangle}{|\nabla \rho|} d H_{N-1} = \int_{S_r} P_k \frac{\langle \nh P_h,\nh \rho\rangle}{|\nabla \rho|} d H_{N-1},
\end{equation}
but, as we next show, \eqref{true} is a far cry from its Euclidean counterpart containing the Euler vector field and the Euclidean norm. To understand this comment we recall \cite[Lemma 6.8]{GR} (see also \cite[formula (2.22)]{GL} for $\Hn$), that states that when $\bG$ is a group of Heisenberg type, with logarithmic coordinates $g = (z,\sigma)$, then for $f\in C^1(\bG)$ one has
\begin{equation}\label{screwy}
\langle \nh f,\nh \rho\rangle =  \frac{Zf}{\rho} |\nabla_H \rho|^2  + \frac{4}{\rho^3} \sum_{\ell=1}^{m_2} \sigma_\ell \Theta_\ell(f),
\end{equation}
where
\[
\Theta_\ell = \sum_{i<j} b^\ell_{ij} \left(z_i \p_{z_j} -  z_j \p_{z_i}\right).
\]
The vector fields $\Theta_\ell$, which come from the complex structure of $\bG$, are the reason for the failure of 
\eqref{wrong2}, and in view of \eqref{difference} also of the failure of the nondecreasing character of Theorems \ref{T:babyACF} and \ref{T:babyC} if we change $|\nhh f|^2$ into $|\nh f|^2$. In view of \eqref{screwy}, when $\bG$ is of Heisenberg type we obtain from \eqref{true}
\begin{equation}\label{noperp}
(k-h) \int_{S_1} P_h P_k |\nh \rho| d\sigma_H = 4 \sum_{\ell=1}^{m_2} \int_{S_1}  \sigma_\ell \big\{P_h \Theta_\ell(P_k) - P_k \Theta_\ell(P_h)\big\} \frac{d H_{N-1}}{|\nabla \rho|}, 
\end{equation}
but it is not true that the right-hand side of \eqref{noperp} generically vanishes when $h\not= k$. This lack of orthogonality of the spherical harmonics causes the nondecreasing monotonicity of the frequency \eqref{fre} to fail for a harmonic function of the type $f = P_h + P_k$. As a consequence, one cannot expect an Almgren type monotonicity formula on a generic harmonic function $f$, unless additional assumptions are imposed on $f$ itself.

We close by illustrating this claim. Suppose that $\bG = \mathbb H^1$ and consider either one of the harmonic functions in $\mathbb H^1$ given in \eqref{fmia} or \eqref{f} above. If to fix the ideas we consider \eqref{f}, since $f = P_1 + P_3$, where $P_1(x,y,\sigma) = x$ and $P_3(x,y,\sigma) = 6 y \sigma - x^3$, with $Z = x \p_x + y \p_y + 2 \sigma \p_\sigma$ we presently have $Z P_1 = P_1$, $Z P_3 = 3 P_3$. As a consequence,  \eqref{fre3} gives 
\[
N(f,r) = 1 + 2 \frac{\int_{S_r} f P_2 |\nh \rho| d\sigma_H}{\int_{S_r} f^2 |\nh \rho| d\sigma_H} = 1 + 2\ \mathscr E(r),
\]
where we have let 
\begin{equation}\label{Eheis}
\mathscr E(r) = \frac{\int_{S_r} f P_3 |\nh \rho| d\sigma_H}{\int_{S_r} f^2 |\nh \rho| d\sigma_H} = \frac{\int_{S_1} f(\delta_r g) P_3(\delta_r g) |\nh \rho| d\sigma_H}{\int_{S_1} f(\delta_r g)^2 |\nh \rho| d\sigma_H}.
\end{equation} 
Observe now that 
\[
f(\delta_r g) P_3(\delta_r g) = (r P_1(g) + r^3 P_3(g)) r^3 P_3(g) = r^4 P_1(g) P_3(g) + r^6 P_3(g)^2,
\]
and
\[
f(\delta_r g)^2 = (r P_1(g) + r^3 P_3(g))^2 = r^2 P_1(g)^2
+ 2 r^4 P_1(g) P_3(g) + r^6 P_3(g)^2.
\]
Now notice that $P_1 P_3 = 6 xy \sigma - x^4$. Since $xy \sigma$ is odd, if we set
\[
a = \int_{S_1} P_1^2 |\nh \rho| d\sigma_H,\ \ \ b = \int_{S_1} x^4 |\nh \rho| d\sigma_H,\ \ \ c = \int_{S_1} P_3^2 |\nh \rho| d\sigma_H,
\]
then $a, b, c >0$, and we have from \eqref{Eheis}
\[
\mathscr E(r) = \frac{- b r^4 + c r^6}{a r^2 - 2 b r^4 + c r^6}  = \frac{- b r^2 + c r^4}{a  - 2 b r^2 + c r^4}.
\]
A simple calculation gives
\[
\mathscr E'(r) = - 2 r \frac{ab + 2 ac r^2 - bc r^4}{(a  - 2 b r^2 + c r^4)^2} \le 0,
\]
provided that $0\le r\le r_0$, for some $r_0>0$ sufficiently small.
Therefore, $r\to N(f,r)$ is nonincreasing on $(0,r_0)$, instead on nondecreasing!


\bibliographystyle{amsplain}

\end{document}